\title{Non-existence results for the weighted $p$-Laplace equation with singular nonlinearities}
\author{Kaushik Bal
        \and        
        Prashanta Garain
        }
\documentclass{article}

\usepackage{graphicx}
\usepackage{hyperref,amsfonts,amsmath,amsthm,bigints}
\usepackage{enumitem,bigints}
\usepackage[usenames, dvipsnames]{color}
\newtheorem{The}{Theorem}[section]
\newtheorem{Lem}{Lemma}[section]
\newtheorem{Rem}{Remark}[section]

\newtheorem{Def}{Definition}[section]

\newenvironment{AMS}{}{}
\newenvironment{keywords}{}{}

\begin{document}
\newpage
\maketitle
\begin{abstract}
In this paper we present some non existence results concerning the stable solutions to the equation 
$$\operatorname{div}(w(x)|\nabla u|^{p-2}\nabla u)=g(x)f(u)\;\;\mbox{in}\;\;\mathbb{R}^N;\;\;p\geq 2$$ 
when $f(u)$ is either $u^{-\delta}+u^{-\gamma}$, $\delta,\gamma>0$ or $\exp(\frac{1}{u})$
and for a suitable class of weight functions $w,g$.
\end{abstract}



\section{Introduction}
Consider the problem 
\begin{equation}\label{mp}
-\operatorname{div}\big(w(x)|\nabla u|^{p-2}\nabla u\big)=g(x)f(u)\;\;\mbox{in}\;\;\mathbb{R}^N\;\;u>0\;\mbox{in}\;\;\mathbb{R}^N
\end{equation}
where $f(u)$ is either $-u^{-\delta}-u^{-\gamma}$ to be denoted by $(\ref{mp})_s$ or $-\exp(\frac{1}{u})$ to be denoted by $(\ref{mp})_e$ and $\delta,\gamma>0$.
We also assume $p\geq 2$, $N\geq 1$ and the weight functions $w,g\in L^1_{loc}(\mathbb{R}^N)$ both positive a.e. in 
$\mathbb{R}^N$ such that $g^{-1}\in L^{\infty}(\mathbb{R}^N)$ unless otherwise mentioned.

\begin{Def}
We say that $u\in C^{1}(\mathbb{R}^N)$ is a weak solution to the equation $(\ref{mp})$ if 
for all $\varphi\in C_{c}^{1}(\mathbb{R}^N)$ we have,
\begin{equation}\label{ws1}
\int_{\mathbb{R}^N}w(x)|\nabla u|^{p-2} \nabla u\nabla \varphi dx=\int_{\mathbb{R}^N}g(x)f(u)\varphi dx
\end{equation}
\end{Def}
\begin{Def}
A weak solution $u$ of equation $(\ref{mp})$ is said to be stable if for all $\varphi\in C_{c}^{1}(\mathbb R^N)$ we have,
\begin{multline}\label{ws}
\int_{\mathbf{R}^N}w(x)|\nabla u|^{p-2}|\nabla \varphi|^2 dx+(p-2)\int_{\mathbf{R}^N}w(x)|\nabla u|^{p-4}(\nabla u,\nabla \varphi)^2 dx\\-\int_{\mathbf{R}^N}g(x)f'(u)\varphi^{2}dx\geq 0
\end{multline}
Therefore if $u$ is a stable solution of equation $(\ref{mp})$ then we have,
\begin{equation}\label{ss}
\int_{\mathbf{R}^N}g(x)f'(u)\varphi^{2}dx\leq(p-1)\int_{\mathbb{R}^N}w(x)|\nabla u|^{p-2}|\nabla \varphi|^2 dx.
\end{equation}
\end{Def}

There has been a recent surge of interest to obtain non-existence results for stable solutions of $C^1$ class for the equation (\ref{mp}) starting from the influential papers of Farina \cite{Fa,Fa1,Fa2} and that of Wei-Ma \cite{JWei}. 
Those papers were followed by some other results concering singular and exponential nonlinearities which can be found in \cite{ChSoYa,DuGuo,GuoMei,Ph,HuYe,KP} and the reference therein.
\section{Notation}
\begin{enumerate}[label=\roman*)]
 \item The equation (\ref{mp}) will be denoted as $(n)_s$ and $(n)_e$ for $f(u)=-u^{-\delta}-u^{-\gamma}$ with $\delta,\gamma>0$ and $f(u)=-e^\frac{1}{u}$ respectively with $n=1,2,3,4$ and $B_{r}(0)$ is the ball centered at $0$ with radius $r>0$
 \item $c$ as a generic constant whose values may vary depending on the situation. If $c$ depends on $\epsilon$ we denote it by $c_\epsilon$
 \item For $p_1$, $p_2>0$ such that $p_1\geq 1$, by $||w||^{p_1}_{L^{p_1}(B_{2r}(0))}=o(r^{p_2})$, we mean there exist a fixed constant $c>0$ independent of $r$ such that $||w||^{p_1}_{L^{p_1}(B_{2r}(0))}\leq cr^{p_2}$
 \item  We say that $(\delta,p)$ belong to the class
 \begin{enumerate}
  \item $P_a$ if either $\delta\geq 1$ for $2\leq p<3$ or $\delta>1$ for $p=3$ or $\delta>\frac{(p-1)^2}{4}$ for $p>3$ respectively.
  \item $P_b$ if either $\delta>0$ for $p=2$ or $\delta>\frac{(p-1)^2}{4}$ for $p>2$ respectively.
 \end{enumerate}
\item Denote $s_p=\delta+\sqrt{\delta^{2}+\delta}$ and $\frac{2\delta}{(p-1)}-\frac{p-1}{2}$ for $p=2$ and $p>2$ respectively. Note that if $(\delta,p)$ belong to either $P_a$ or $P_b$, then we have $s_p>0$
\item Denote $S_p^{a}=1+\frac{2s_p}{p-1+\delta}$ and $S_p^{b}=1+\frac{2s_p}{p-1+\gamma}$ for $p\geq 2$
\item Denote $m=||g^{-1}||_{\infty}$
\item Let us fix a constant $M$ such that $M>0$ for $p=2$ and $0<M<\frac{4}{(p-1)^2}$ for $p>2$ respectively. Denote
\begin{equation}
 t_p=
 \begin{cases}
\frac{1}{M}+\sqrt{\frac{1}{M}+\frac{1}{M^2}}\;\mbox{for}\;p=2\\
 \frac{2}{M(p-1)}-\frac{p-1}{2}\;\mbox{for}\;p>2
 \end{cases}
 \end{equation}

\item Define the number $T_p=1+\frac{2t_p}{p}$ for $p\geq 2$. Observe that $t_p>0$ for any $p\geq 2$ and therefore $T_p>1$
\end{enumerate}

\section{Main Results}\label{mr}
We start this section by stating our main results
\begin{The}(\textbf{Cacciopolli type Estimate})\label{smain1}
Let $u\in C^{1}(\mathbb{R}^N)$ be a positive stable solution to the problem $(\ref{mp})_s$.  
Then the following holds:\\
(A) Let $(\delta,p)$ belong to the class $P_a$ such that $\delta<\gamma$. Assume that $0<u\leq 1$ in $\mathbb{R}^N$. Then for any $\beta\in(0,s_p)$, there exists a constant $c=c(\beta,p,m)>0$ 
such that for every $\psi\in C_c^{1}(\mathbb{R}^N)$ with $0\leq\psi\leq 1$ in $\mathbb{R}^N$, we have 
\begin{equation}\label{smineq1}
\int_{\mathbb{R}^N} g(x)(\frac{\psi}{u})^{2\beta+p-1+\delta} \,dx\leq c\int_{\mathbb{R}^N} w^{\theta_{a}'}|\nabla\psi|^{p{\theta_{a}'}} \,dx
\end{equation}
where $\theta_a=\frac{2\beta+p-1+\delta}{2\beta}$ and $\theta_{a}'=\frac{2\beta+p-1+\delta}{p-1+\delta}$.\\
(B) Let $(\delta,p)$ belong to the class $P_b$ such that $\delta<\gamma$ where $\gamma\geq 1$. Assume that $u\geq 1$ in $\mathbb{R}^N$. Then for any $\beta\in(0,s_p)$, there exists a constant $c=c(\beta,p,m)>0$ 
such that for every $\psi\in C_c^{1}(\mathbb{R}^N)$ with $0\leq\psi\leq 1$ in $\mathbb{R}^N$, we have 
\begin{equation}\label{smineq2}
\int_{\mathbb{R}^N} g(x)(\frac{\psi}{u})^{2\beta+p-1+\gamma} \,dx\leq c\int_{\mathbb{R}^N} w^{\theta_{b}'}|\nabla\psi|^{p{\theta_{b}'}} \,dx
\end{equation}
where $\theta_b=\frac{2\beta+p-1+\gamma}{2\beta}$ and $\theta_{b}'=\frac{2\beta+p-1+\gamma}{p-1+\gamma}$.\\ 
(C) Let $(\delta,p)$ belong to the class $P_a$ such that $\delta=\gamma\geq 1$. Assume that $u>0$ in $\mathbb{R}^N$. Then for any $\beta\in(0,s_p)$, there exists a constant $c=c(\beta,p,m)>0$
such that for every $\psi\in C_c^{1}(\mathbb{R}^N)$ with $0\leq\psi\leq 1$ in $\mathbb{R}^N$, we have 
\begin{equation}\label{smineq3}
\int_{\mathbb{R}^N} g(x)(\frac{\psi}{u})^{2\beta+p-1+\delta} \,dx\leq c\int_{\mathbb{R}^N} w^{\theta_{a}'}|\nabla\psi|^{p{\theta_{a}'}} \,dx
\end{equation}
where $\theta_a=\frac{2\beta+p-1+\delta}{2\beta}$ and $\theta_{a}'=\frac{2\beta+p-1+\delta}{p-1+\delta}$. 
\end{The}

\begin{The}\label{smain2}
Let $(\delta,p)$ belong to the class $P_a$ such that $\delta<\gamma$ and $u\in C^1(\mathbb{R}^N)$ such that $0<u\leq 1$ in $\mathbb{R}^N$. Assume that there exist a fixed constant $\lambda_p^{a}>0$
such that $$||w||^{S_p^{a}}_{L^{S_p^{a}}(B_{2R}(0))}=o(R^{\lambda_p^{a}})\;\;\mbox{and}\;\;\lambda_p^{a}<pS_p^{a}.$$\\
Then $u$ is not a stable solution to the problem $(\ref{mp})_s$.
\end{The}

\begin{The}\label{smain3}
Let $(\delta,p)$ belong to the class $P_b$ such that $\delta<\gamma$ with $\gamma\geq 1$ and $u\in C^1(\mathbb{R}^N)$ such that $u\geq 1$ in $\mathbb{R}^N$. 
Assume that there exist a fixed constant $\lambda_p^{b}>0$ such that $$||w||^{S_p^{b}}_{L^{S_p^{b}}(B_{2R}(0))}=o(R^{\lambda_p^{b}})\;\;\mbox{and}\;\;\lambda_p^{b}<pS_p^{b}.$$\\
Then $u$ is not a stable solution to the problem $(\ref{mp})_s$. 
\end{The}

\begin{The}\label{smain4}
Let $(\delta,p)$ belong to the class $P_a$ such that $\delta=\gamma\geq 1$ and $u\in C^1(\mathbb{R}^N)$ such that $u>0$ in $\mathbb{R}^N$. Assume that there exist a fixed constant $\lambda_p^{c}>0$ 
such that $$||w||^{S_p^{a}}_{L^{S_p^{a}}(B_{2R}(0))}=o(R^{\lambda_p^{c}})\;\;\mbox{and}\;\;\lambda_p^{c}<pS_p^{a}.$$
Then $u$ is not a stable solution to the problem $(\ref{mp})_s$. 
\end{The}

\begin{Rem}\label{rem1}
For $w\in L^\infty(\mathbb{R}^N)$, we can choose $\lambda_p^{a},\lambda_p^{b},\lambda_p^{c}$ equal to the dimension $N$ of $\mathbb{R}^N$ in the above theorems $(\ref{smain2}),(\ref{smain3}),(\ref{smain4})$.
\end{Rem}

\begin{The}(\textbf{Cacciopolli type Estimate})\label{Main}
Let $u\in C^{1}(\mathbb{R}^N)$ be a bounded positive stable solution to the problem $(\ref{mp})_e$ such that $||u||_{L^{\infty}(\mathbb{R}^N)}\leq M$. Then for any $\beta\in(0,t_p)$, 
there exists a constant $c=c(\beta,p,m)>0$ such that for every $\psi\in C_c^{1}(\mathbb{R}^N)$ with $0\leq\psi\leq 1$ in $\mathbb{R}^N$, we have 
\begin{equation}\label{ineq1}
\int_{\mathbb{R}^N} g(x)(\frac{\psi}{u})^{2\beta+p} \,dx\leq c\int_{\mathbb{R}^N} w^{\theta'}|\nabla\psi|^{p{\theta'}} \,dx
\end{equation}
where $\theta=\frac{2\beta+p}{2\beta}$ and $\theta'=\frac{2\beta+p}{p}$
\end{The}

\begin{The}\label{Main1}
Let $u\in C^{1}(\mathbb{R}^N)$ be positive such that $||u||_{L^\infty(\mathbb{R}^N)}\leq M$. 
%(I) Let $p=2$ and assume that there exist a fixed constant $\mu_2>0$ such that $$||w||^{T_2}_{L^{T_2}(B_{2R}(0))}=o(R^{\mu_2}),$$ and $\mu_2<2T_2.$\\
Assume that there exist a fixed constant $\mu_p>0$ such that $$||w||^{T_p}_{L^{T_p}(B_{2R}(0))}=o(R^{\mu_p})\;\;\mbox{and}\;\;\mu_p<pT_p$$\\
Then $u$ is not a stable solution to the problem $(\ref{mp})_e$.
\end{The}

\begin{Rem}\label{Cor1}
For $w\in L^{\infty}(\mathbb{R}^N)$, we may choose $\mu_p=N$ in theorem $(\ref{Main1})$.
\end{Rem}

\section{Proof of Main Theorems}
Before proving the results of Section \ref{mr}, we prove an important lemma.
\begin{Lem}\label{Lemma}
Assume $p\geq 2$ and let $u\in C^{1}(\mathbb{R}^N)$ be a positive stable solution to the problem $(\ref{mp})$. Fix a constant $\beta>0$. Then for every $\epsilon\in(0,2\beta+p-1)$, 
there exist a constant $c_\epsilon=c_{\epsilon}(\beta,p)>0$ such that for any nonnegative $\psi\in C_{c}^{1}(\mathbb{R}^N)$, we have  
\begin{multline}\label{e15}
\int_{\mathbb{R}^N} g(x)f'(u)u^{-2\beta-p+2}\psi^p \,dx
\leq c_\epsilon\int_{\mathbb{R}^N} w(x)u^{-2\beta}|\nabla\psi|^{p} \,dx\\
-\frac{(p-1)(\beta+\frac{p}{2}-1)^2+\epsilon}{2\beta+p-1-\epsilon}\int_{\mathbb{R}^N} g(x)f(u)u^{-2\beta-p+1}\psi^p \,dx 
\end{multline}
\end{Lem}
\begin{proof}[Proof of Lemma \ref{Lemma}]
Suppose $u\in C^{1}(\mathbb{R}^N)$ be a positive stable solution of equation $(\ref{mp})$ and $\psi\in C_c^{1}(\mathbb{R}^N)$ be nonnegative. We prove the lemma in two steps.\\
\textbf{Step 1.} 
Suppose $u\in C^1(\mathbb{R}^N)$ be a positive weak solution of $(\ref{mp})$ and let $\alpha>0$.
Choosing $\varphi=u^{-\alpha}\psi^p$ as a test function in the weak form $(\ref{ws1})$, since
$$
\nabla\varphi=-\alpha u^{-\alpha-1}\psi^p\nabla u+p\psi^{p-1}u^{-\alpha}\nabla\psi
$$
we obtain 
\begin{multline}\label{e11}
\alpha\int_{\mathbb{R}^N} w(x)u^{-\alpha-1}\psi^p|\nabla u|^{p} \,dx
\leq p\int_{\mathbb{R}^N} w(x)u^{-\alpha}\psi^{p-1}|\nabla u|^{p-1}|\nabla\psi| \,dx\\
-\int_{\mathbb{R}^N} g(x)u^{-\alpha}f(u)\psi^p \,dx
\end{multline}
Now using the Young's inequality for $\epsilon\in(0,\alpha)$, we obtain
\begin{multline*}
p\int_{\mathbb{R}^N} w(x)u^{-\alpha}\psi^{p-1}|\nabla u|^{p-1}|\nabla\psi| \,dx\\
=p\int_{\mathbb{R}^N}(w^\frac{1}{{p}'}u^{\frac{-\alpha-1}{{p}'}}|\nabla u|^{p-1}\psi^{p-1})(w^\frac{1}{p}u^\frac{p-\alpha-1}{p}|\nabla\psi|)dx\\
\leq \{\epsilon\int_{\mathbb{R}^N} w(x)u^{-\alpha-1}\psi^p|\nabla u|^{p} \,dx\\+c_{\epsilon}\int_{\mathbb{R}^N} w(x)u^{p-\alpha-1}|\nabla\psi|^{p}\} \,dx
\end{multline*}
Plugging this estimate in $(\ref{e11})$ and defining
$$
A=\int_{\mathbb{R}^N} w(x)u^{-\alpha-1}\psi^p|\nabla u|^{p} \,dx,
$$
 
$$
B=\int_{\mathbb{R}^N} w(x)u^{p-\alpha-1}|\nabla\psi|^{p} \,dx
$$
we obtain
\begin{equation}\label{e12}
(\alpha-\epsilon)A\leq c_{\epsilon}B-\int_{\mathbb{R}^N} g(x)u^{-\alpha}f(u)\psi^p \,dx
\end{equation}

\textbf{Step 2.} Suppose $u\in C^1(\mathbb{R}^N)$ be a positive stable solution of $(\ref{mp})$ and let $\beta>0$.
Choosing $\varphi=u^{-\beta-\frac{p}{2}+1}\psi^{\frac{p}{2}}$, as a test function in the stability equation $(\ref{ss})$, since
$$
\nabla\varphi=(-\beta-\frac{p}{2}+1)u^{-\beta-\frac{p}{2}}\psi^{\frac{p}{2}}\nabla u+\frac{p}{2}\psi^{\frac{p-2}{2}}u^{-\beta-\frac{p}{2}+1}\nabla\psi
$$
we obtain
\begin{multline*}
\int_{\mathbb{R}^N} g(x)f'(u)u^{-2\beta-p+2}\psi^p \,dx
\leq(p-1)(-\beta-\frac{p}{2}+1)^2\int_{\mathbb{R}^N} w(x)u^{-2\beta-p}\psi^p |\nabla u|^{p} \,dx\\
+(p-1)\frac{p^2}{4}\int_{\mathbb{R}^N} w(x)u^{-2\beta-p+2}\psi^{p-2}|\nabla u|^{p-2}|\nabla\psi|^2 \,dx\\+
(p-1)(-\beta-\frac{p}{2}+1)p\int_{\mathbb{R}^N} w(x)u^{-2\beta-p+1}\psi^{p-1}|\nabla u|^{p-1}|\nabla\psi| \,dx 
=(X+Y+Z)
\end{multline*}
where 
$$
X=(p-1)(-\beta-\frac{p}{2}+1)^2\int_{\mathbb{R}^N} w(x)u^{-2\beta-p}\psi^p |\nabla u|^{p} \,dx
$$
$$
Y=(p-1)\frac{p^2}{4}\int_{\mathbb{R}^N} w(x)u^{-2\beta-p+2}\psi^{p-2}|\nabla u|^{p-2}|\nabla\psi|^2 \,dx
$$
and 
$$
Z=(p-1)(-\beta-\frac{p}{2}+1)p\int_{\mathbb{R}^N} w(x)u^{-2\beta-p+1}\psi^{p-1}|\nabla u|^{p-1}|\nabla\psi| \,dx
$$
Therefore we have
\begin{equation}\label{e13}
\int_{\mathbb{R}^N} g(x)f'(u)u^{-2\beta-p+2}\psi^p \,dx\leq(X+Y+Z).
\end{equation}
Let $\alpha=2\beta+p-1$. Then we have $X=(\beta+\frac{p}{2}-1)^2(p-1)A$. Now we prove the estimate for $p=2$ and $p>2$ seperately in case a and case b.\\ 
\textbf{Case a.} Let $p=2$. Therefore
\begin{align*}
Y &=\int_{\mathbb{R}^N} w(x)u^{-2\beta}|\nabla\psi|^2 dx=B
\end{align*}
Also using the exponents $p{'}=\frac{p}{p-1}$ and $p$ in the Young's inequality, we obtain
\begin{align*}
Z&=-2\beta\int_{\mathbb{R}^N} w(x)u^{-2\beta-1}\psi|\nabla u||\nabla\psi| \,dx\\
&=-2\beta\int_{\mathbb{R}^N} (w^{\frac{1}{2}}\psi u^{-\beta-1}|\nabla u|)(w^{\frac{1}{2}}u^{-\beta}|\nabla\psi|) \,dx\\
&\leq\epsilon\int_{\mathbb{R}^N} w(x)u^{-2\beta-2}\psi^2|\nabla u|^{2} \,dx+c_\epsilon\int_{\mathbb{R}^N} w(x)u^{-2\beta}|\nabla\psi|^{2} \,dx\\
&=\epsilon A+c_\epsilon B.
\end{align*}
Now putting $\alpha=2\beta+1>0$ in equation (\ref{e12}) we get
\begin{equation}\label{eI}
A\leq\frac{1}{2\beta+1-\epsilon}\{c_\epsilon B-\int_{\mathbb{R}^N} g(x)f(u)u^{-2\beta-1}\psi^2 \,dx\}
\end{equation}
Using the inequality $(\ref{e14})$ together with the above estimates on $Y$ and $Z$ in $(\ref{e13})$, we get
\begin{multline*}
\int_{\mathbb{R}^N} g(x)f'(u)u^{-2\beta}\psi^2 \,dx
\leq c_\epsilon\big(\frac{\beta^2+\epsilon}{2\beta+1-\epsilon}+1+c_{\epsilon}\big) B\\
-\frac{\beta^2+\epsilon}{2\beta+1-\epsilon}\int_{\mathbb{R}^N} g(x)f(u)u^{-2\beta-1}\psi^2 \,dx
\end{multline*}
\textbf{Case b.} Let $p>2$.
and using the exponents $\frac{p}{p-2}$ and $\frac{p}{2}$ in the Young's inequality, we have
\begin{align*}
Y &=(p-1)\frac{p^2}{4}\int_{\mathbb{R}^N} w(x)u^{-2\beta-p+2}\psi^{p-2}|\nabla u|^{p-2}|\nabla\psi|^2 dx\\
&=(p-1)\frac{p^2}{4}\int_{\mathbb{R}^N}(w^\frac{p-2}{p}|\nabla u|^{p-2}\psi^{p-2}u^\frac{(-2\beta-p)(p-2)}{p})(w^\frac{2}{p}u^\frac{2(-2\beta)}{p}|\nabla\psi|^2) \,dx\\
&\leq\frac{\epsilon}{2}\int_{\mathbb{R}^N} w(x)u^{-2\beta-p}\psi^p|\nabla u|^{p} \,dx+\frac{c_\epsilon}{2}\int_{\mathbb{R}^N} w(x)u^{-2\beta}|\nabla\psi|^{p} \,dx\\
&=\frac{\epsilon}{2} A+\frac{c_\epsilon}{2} B
\end{align*}
Also using the exponents $p{'}=\frac{p}{p-1}$ and $p$ in the Young's inequality, we obtain
\begin{align*}
Z&=(p-1)(-\beta-\frac{p}{2}+1)p\int_{\mathbb{R}^N} w(x)u^{-2\beta-p+1}\psi^{p-1}|\nabla u|^{p-1}|\nabla\psi| \,dx\\
&=(p-1)(-\beta-\frac{p}{2}+1)p\int_{\mathbb{R}^N} (w^{\frac{1}{p'}}\psi^{p-1}u^y|\nabla u|^{p-1})(w^{\frac{1}{p}}u^{y'}|\nabla\psi|) \,dx\\
&\leq\frac{\epsilon}{2}\int_{\mathbb{R}^N} w(x)u^{-2\beta-p}\psi^p|\nabla u|^{p} \,dx+\frac{c_\epsilon}{2}\int_{\mathbb{R}^N} w(x)u^{-2\beta}|\nabla\psi|^{p} \,dx\\
&=\frac{\epsilon}{2} A+\frac{c_\epsilon}{2} B.
\end{align*}
where $y=\frac{(-2\beta-p)(p-1)}{p}$ and $y'={-2\beta-p+1+\frac{(2\beta+p)(p-1)}{p}}$
Now putting $\alpha=2\beta+p-1>0$ in equation (\ref{e12}) we get 
\begin{equation}\label{e14}
A\leq\frac{1}{2\beta+p-1-\epsilon}\{c_\epsilon B-\int_{\mathbb{R}^N} g(x)f(u)u^{-2\beta-p+1}\psi^p \,dx\}
\end{equation}
Using the inequality $(\ref{e14})$ together with the above estimates on $Y$ and $Z$ in $(\ref{e13})$, we get
\begin{multline*}
\int_{\mathbb{R}^N} g(x)f'(u)u^{-2\beta-p+2}\psi^p \,dx
\leq c_\epsilon\big(\frac{(p-1)(\beta+\frac{p}{2}-1)^2+\epsilon}{2\beta+p-1-\epsilon}+1\big) B\\
-\frac{(p-1)(\beta+\frac{p}{2}-1)^2+\epsilon}{2\beta+p-1-\epsilon}\int_{\mathbb{R}^N} g(x)f(u)u^{-2\beta-p+1}\psi^p \,dx
\end{multline*}
Hence the lemma follows.
\end{proof}

\begin{proof}[Proof of Theorem \ref{smain1}]
Let $u\in C^{1}(\mathbb{R}^N)$ be a positive stable solution to the problem $(\ref{mp})_s$. Then using the fact $0<\delta\leq\gamma$ and $f(u)=-u^{-\delta}-u^{-\gamma}$ in the inequality $(\ref{e15})$, we obtain
\begin{multline*}
\beta_\epsilon\int_{\mathbb{R}^N} g(x)\big(\frac{1}{u^\delta}+\frac{1}{u^\gamma}\big)u^{-2\beta-p+1}\psi^p \,dx\\
\leq c_\epsilon\int_{\mathbb{R}^N} w(x)u^{-2\beta}|\nabla \psi|^p dx
\end{multline*}
where $\beta_{\epsilon}=\big(\delta-\frac{(p-1)(\beta+\frac{p}{2}-1)^2+\epsilon}{2\beta+p-1-\epsilon}\big)$. Therefore, we have 
$\lim\limits_{\epsilon\to 0}\beta_{\epsilon}=\big(\delta-\frac{(p-1)(\beta+\frac{p}{2}-1)^2}{2\beta+p-1}\big)>0$ for every $\beta\in(0,s_p)$. So we can choose an $\epsilon\in(0,1)$ such that $\beta_{\epsilon}>0$
Hence we have
\begin{equation}\label{new}
\int_{\mathbb{R}^N} g(x)\big(\frac{1}{u^\delta}+\frac{1}{u^\gamma}\big)u^{-2\beta-p+1}\psi^p \,dx
\leq c\int_{\mathbb{R}^N} w(x)u^{-2\beta}|\nabla \psi|^p dx.
\end{equation}
(A) Since $\delta<\gamma$ and $0<u\leq 1$ in $\mathbb{R^N}$, we have for any $\beta\in(0,s_p)$, the inequality $(\ref{new})$ becomes
\begin{align*}
\int_{\mathbb{R}^N} g(x)u^{-2\beta-p+1-\delta}\psi^p \,dx
\leq c\int_{\mathbb{R}^N} w(x)u^{-2\beta}|\nabla \psi|^p dx
\end{align*}
Replacing $\psi$ by $\psi^\frac{2\beta+p-1+\delta}{p}$, we obtain
\begin{align*}
\int_{\mathbb{R}^N} g(x)(\frac{\psi}{u})^{2\beta+p-1+\delta}
&\leq c(\frac{2\beta+p-1+\delta}{p})^p\int_{\mathbb{R}^N}w(x)u^{-2\beta}\psi^{2\beta+\delta-1}|\nabla\psi|^p dx\\
&=c\int_{\mathbb{R}^N}(\frac{\psi}{u})^{2\beta}(w(x)|\nabla\psi|^p)\psi^{\delta-1} \,dx
\end{align*}
Choosing the exponents $\theta_a=\frac{2\beta+p-1+\delta}{2\beta}, \theta_{a}'=\frac{2\beta+p-1+\delta}{p-1+\delta}$ in the Young's inequality we get,
\begin{multline*}
\int_{\mathbb{R}^N} g(x)(\frac{\psi}{u})^{2\beta+p-1+\delta} \,dx 
\leq\{\epsilon\int_{\mathbb{R}^N} g(x)(\frac{\psi}{u})^{2\beta+p-1+\delta}dx\\+c_{\epsilon}\int_{\mathbb{R}^N}w^{\theta_{a}'}g^{-\frac{\theta_{a}'}{\theta_{a}}}\psi^{(\delta-1)\theta_{a}'}|\nabla\psi|^{p\theta_{a}'}dx\}
\end{multline*}
Using the fact $g^{-1}\in L^\infty(\mathbb{R}^N)$, $0\leq\psi\leq 1$ and $\delta\geq 1$, we get the inequality
\begin{align*}
\int_{\mathbb{R}^N} g(x)(\frac{\psi}{u})^{2\beta+p-1+\delta} \,dx
\leq c\int_{\mathbb{R}^N} w^{\theta_{a}'}|\nabla\psi|^{p\theta_{a}'}dx
\end{align*}

(B) Since $\delta<\gamma$ and $u\geq 1$ in $\mathbb{R^N}$, we have for any $\beta\in(0,s_p)$, the inequality $(\ref{new})$ becomes
\begin{align*}
\int_{\mathbb{R}^N} g(x)u^{-2\beta-p+1-\gamma}\psi^p \,dx
\leq c\int_{\mathbb{R}^N} w(x)u^{-2\beta}|\nabla \psi|^p dx
\end{align*}
Replacing $\psi$ by $\psi^\frac{2\beta+p-1+\gamma}{p}$ we obtain
\begin{align*}
\int_{\mathbb{R}^N} g(x)(\frac{\psi}{u})^{2\beta+p-1+\gamma}
&\leq c(\frac{2\beta+p-1+\gamma}{p})^p\int_{\mathbb{R}^N}w(x)u^{-2\beta}\psi^{2\beta+\gamma-1}|\nabla\psi|^p dx\\
&=c\int_{\mathbb{R}^N}(\frac{\psi}{u})^{2\beta}(w(x)|\nabla\psi|^p)\psi^{\gamma-1} \,dx
\end{align*}
Choosing the exponents $\theta_b=\frac{2\beta+p-1+\gamma}{2\beta}, \theta_{b}'=\frac{2\beta+p-1+\gamma}{p-1+\gamma}$ in the Young's inequality we get,
\begin{multline*}
\int_{\mathbb{R}^N} g(x)(\frac{\psi}{u})^{2\beta+p-1+\gamma} \,dx 
\leq\{\epsilon\int_{\mathbb{R}^N} g(x)(\frac{\psi}{u})^{2\beta+p-1+\gamma}dx\\+c_{\epsilon}\int_{\mathbb{R}^N}w^{\theta_{b}'}g^{-\frac{\theta_{b}'}{\theta_{b}}}\psi^{(\gamma-1)\theta_{b}'}|\nabla\psi|^{p\theta_{b}'}dx\}
\end{multline*}
Using the fact $g^{-1}\in L^\infty(\mathbb{R}^N)$, $0\leq\psi\leq 1$ and $\gamma\geq 1$, we get the inequality
\begin{equation}\label{Cac2}
\int_{\mathbb{R}^N} g(x)(\frac{\psi}{u})^{2\beta+p-1+\gamma} \,dx
\leq c\int_{\mathbb{R}^N} w^{\theta_{b}'}|\nabla\psi|^{p\theta_{b}'}dx
\end{equation}

(C) Since $\delta=\gamma$ and $u>0$ in $\mathbb{R^N}$, we have for any $\beta\in(0,s_p)$, the inequality $(\ref{new})$ becomes
\begin{align*}
\int_{\mathbb{R}^N} g(x)u^{-2\beta-p+1-\delta}\psi^p \,dx
\leq c\int_{\mathbb{R}^N} w(x)u^{-2\beta}|\nabla \psi|^p dx
\end{align*}
Replacing $\psi$ by $\psi^\frac{2\beta+p-1+\delta}{p}$, we obtain
\begin{align*}
\int_{\mathbb{R}^N} g(x)(\frac{\psi}{u})^{2\beta+p-1+\delta}
&\leq c(\frac{2\beta+p-1+\delta}{p})^p\int_{\mathbb{R}^N}w(x)u^{-2\beta}\psi^{2\beta+\delta-1}|\nabla\psi|^p dx\\
&=c\int_{\mathbb{R}^N}(\frac{\psi}{u})^{2\beta}(w(x)|\nabla\psi|^p)\psi^{\delta-1} \,dx
\end{align*}
Choosing the exponents $\theta_a=\frac{2\beta+p-1+\delta}{2\beta}, \theta_{a}'=\frac{2\beta+p-1+\delta}{p-1+\delta}$ in the Young's inequality we get,
\begin{multline*}
\int_{\mathbb{R}^N} g(x)(\frac{\psi}{u})^{2\beta+p-1+\delta} \,dx 
\leq\{\epsilon\int_{\mathbb{R}^N} g(x)(\frac{\psi}{u})^{2\beta+p-1+\delta}dx\\+c_{\epsilon}\int_{\mathbb{R}^N}w^{\theta_{a}'}g^{-\frac{\theta_{a}'}{\theta_{a}}}\psi^{(\delta-1)\theta_{a}'}|\nabla\psi|^{p\theta_{a}'}dx\}
\end{multline*}
Using the fact $g^{-1}\in L^\infty(\mathbb{R}^N)$, $0\leq\psi\leq 1$ and $\delta\geq 1$, we get the inequality
\begin{align*}
\int_{\mathbb{R}^N} g(x)(\frac{\psi}{u})^{2\beta+p-1+\delta} \,dx
\leq c\int_{\mathbb{R}^N} w^{\theta_{a}'}|\nabla\psi|^{p\theta_{a}'}dx
\end{align*} 
\end{proof}

\begin{proof}[Proof of Theorem \ref{smain2}]
By contradiction, let us suppose that $u$ be a positive stable solution to the problem $(\ref{mp})_s$. Then by part (A) of Theorem $\ref{smain1}$, we have
\begin{align*}
\int_{\mathbb{R}^N} g(x)(\frac{\psi_R}{u})^{2\beta+p-1+\delta} \,dx
\leq c\int_{\mathbb{R}^N} w^{\theta_{a}'}|\nabla\psi_R|^{p\theta_{a}'}dx
\end{align*}  
Choosing $\psi_{R}\in C_{c}^{1}(\mathbb{R}^N)$ such that $0\leq\psi_{R}\leq 1$ in $\mathbb{R}^N$, $\psi_{R}=1$ in $B_{R}(0)$ and $\psi_{R}=0$ in $\mathbb{R}^N\setminus B_{2R}(0)$ with 
$|\nabla\psi_R|\leq\frac{c}{R}$ for $c>1$ in $(\ref{ineq1})$, we obtain
\begin{equation}\label{smain2ineq}
\int_{B_{R}(0)} g(x)(\frac{1}{u})^{2\beta+p-1+\delta} \,dx\leq cR^{-p\theta_{a}'}\int_{B_{2R}(0)} w^{\theta_{a}'} \,dx
\end{equation}
Since $\theta_{a}'=\frac{2\beta+p-1+\delta}{p-1+\delta}<S_p^{a}$ and $||w||^{S_p^{a}}_{L^{S_p^{a}}(B_{2R}(0))}=o(R^{\lambda_p^{a}})$, we have 
$$
\int_{B_{2R}(0)} w^{\theta_{a}'} \,dx\leq c\big(\int_{B_{2R}(0)} w^{S_p^{a}} \,dx\big)^{\frac{\theta_{a}'}{S_p^{a}}}\leq c R^\frac{\theta_{a}'\lambda_p^{a}}{S_p^{a}}
$$ 
Hence from $(\ref{smain2ineq})$, we obtain  
$$
\int_{B_{R}(0)} g(x)(\frac{1}{u})^{2\beta+p-1+\delta} \,dx\leq cR^{{\theta'_a}(\frac{\lambda_p^{a}}{S_p^{a}}-p)}
$$
Since $\lambda_p^{a}<pS_p^{a}$, letting $R\to\infty$, we obtain
$$
\int_{\mathbb{R}^N} g(x)(\frac{1}{u})^{2\beta+p-1+\delta} \,dx=0
$$
which is a contradiction.
\end{proof}

\begin{proof}[Proof of Theorem \ref{smain3}]
By contradiction, let us suppose that $u$ be a positive stable solution to the problem $(\ref{mp})_s$. Then by part (B) of Theorem $\ref{smain1}$, we have
\begin{align*}
\int_{\mathbb{R}^N} g(x)(\frac{\psi_R}{u})^{2\beta+p-1+\gamma} \,dx
\leq c\int_{\mathbb{R}^N} w^{\theta_{b}'}|\nabla\psi_R|^{p\theta_{b}'}dx
\end{align*}  
Choosing $\psi_{R}\in C_{c}^{1}(\mathbb{R}^N)$ such that $0\leq\psi_{R}\leq 1$ in $\mathbb{R}^N$, $\psi_{R}=1$ in $B_{R}(0)$ and $\psi_{R}=0$ in $\mathbb{R}^N\setminus B_{2R}(0)$ 
with $|\nabla\psi_R|\leq\frac{c}{R}$ for $c>1$ in $(\ref{ineq1})$, we obtain
\begin{equation}\label{smain3ineq}
\int_{B_{R}(0)} g(x)(\frac{1}{u})^{2\beta+p-1+\gamma} \,dx\leq cR^{-p\theta_{b}'}\int_{B_{2R}(0)} w^{\theta_{b}'} \,dx
\end{equation}
Since $\theta_{b}'=\frac{2\beta+p-1+\delta}{p-1+\delta}<S_p^{b}$ and $||w||^{S_p^{b}}_{L^{S_p^{b}}(B_{2R}(0))}=o(R^{\lambda_p^{b}})$, we have 
$$
\int_{B_{2R}(0)} w^{\theta_{b}'} \,dx\leq c\big(\int_{B_{2R}(0)} w^{S_p^{b}} \,dx\big)^{\frac{\theta_{b}'}{S_p^{b}}}\leq c R^\frac{\theta_{b}'\lambda_p^{b}}{S_p^{b}}
$$ 
Hence from $(\ref{smain3ineq})$, we obtain  
$$
\int_{B_{R}(0)} g(x)(\frac{1}{u})^{2\beta+p-1+\gamma} \,dx\leq cR^{{\theta_{b}'}(\frac{\lambda_p^{b}}{S_p^{b}}-p)}
$$
Since $\lambda_p^{b}<pS_p^{b}$, letting $R\to\infty$, we obtain
$$
\int_{\mathbb{R}^N} g(x)(\frac{1}{u})^{2\beta+p-1+\gamma} \,dx=0
$$
which is a contradiction.
\end{proof}

\begin{proof}[Proof of Theorem \ref{smain4}]
By contradiction, let us suppose that $u$ be a positive stable solution to the problem $(\ref{mp})_s$. Then by part (C) of Theorem $\ref{smain1}$, we have
\begin{align*}
\int_{\mathbb{R}^N} g(x)(\frac{\psi_R}{u})^{2\beta+p-1+\delta} \,dx
\leq c\int_{\mathbb{R}^N} w^{\theta_{a}'}|\nabla\psi|^{p\theta_{a}'}dx.
\end{align*}  
Choosing $\psi_{R}\in C_{c}^{1}(\mathbb{R}^N)$ such that $0\leq\psi_{R}\leq 1$ in $\mathbb{R}^N$, $\psi_{R}=1$ in $B_{R}(0)$ and $\psi_{R}=0$ in $\mathbb{R}^N\setminus B_{2R}(0)$ 
with $|\nabla\psi_R|\leq\frac{c}{R}$ for $c>1$ in $(\ref{ineq1})$, we obtain
\begin{equation}\label{smain4ineq}
\int_{B_{R}(0)} g(x)(\frac{1}{u})^{2\beta+p-1+\delta} \,dx\leq cR^{-p\theta_{a}'}\int_{B_{2R}(0)} w^{\theta_{a}'} \,dx
\end{equation}
Since $\theta_{a}'=\frac{2\beta+p-1+\delta}{p-1+\delta}<S_p^{a}$ and $||w||^{S_p^{a}}_{L^{S_p^{a}}(B_{2R}(0))}=o(R^{\lambda_p^{c}})$, we have 
$$
\int_{B_{2R}(0)} w^{\theta_{a}'} \,dx\leq c\big(\int_{B_{2R}(0)} w^{S_p^{a}} \,dx\big)^{\frac{\theta_{a}'}{S_p^{a}}}\leq c R^\frac{\theta_{a}'\lambda_p^{c}}{S_p^{a}}
$$ 
Hence from $(\ref{smain4ineq})$, we obtain  
$$
\int_{B_{R}(0)} g(x)(\frac{1}{u})^{2\beta+p-1+\delta} \,dx\leq cR^{{\theta_{a}'}(\frac{\lambda_p}{S_p^{a}}-p)}
$$
Since $\lambda_p^{c}<pS_p^{a}$, letting $R\to\infty$, we obtain
$$
\int_{\mathbb{R}^N} g(x)(\frac{1}{u})^{2\beta+p-1+\delta} \,dx=0
$$
which is a contradiction.
\end{proof}

\begin{proof}[Proof of Theorem \ref{Main}]
Let $u\in C^{1}(\mathbb{R}^N)$ be a bounded positive stable solution to the problem $(\ref{mp})_e$ such that $||u||_{L^\infty(\mathbb{R}^N)}\leq M$.
Using the condition $0<u\leq M$ in $\mathbb{R}^N$ and $f(u)=-e^\frac{1}{u}$ in the inequality $(\ref{e15})$, we obtain
\begin{align*}
\beta_\epsilon\int_{\mathbb{R}^N} g(x)e^\frac{1}{u}u^{-2\beta-p+1}\psi^p \,dx
\leq c_{\epsilon}\int_{\mathbb{R}^N} w(x)u^{-2\beta}|\nabla \psi|^p dx.
\end{align*}
Define $\beta_{\epsilon}=\big(\frac{1}{M}-\frac{(p-1)(\beta+\frac{p}{2}-1)^2+\epsilon}{2\beta+p-1-\epsilon}\big)$. Therefore, we have $\lim\limits_{\epsilon\to 0}\beta_{\epsilon}=\big(\frac{1}{M}-\frac{(p-1)(\beta+\frac{p}{2}-1)^2}{2\beta+p-1}\big)>0$ 
for every $\beta\in(0,t_p)$.\\
Therefore we can choose an $\epsilon>0$ such that $\beta_{\epsilon}>0$.
Hence using the inequality $e^\frac{1}{u}>\frac{1}{u}$, we obtain
\begin{equation*}
\int_{\mathbb{R}^N} g(x)u^{-2\beta-p}\psi^p \,dx\leq c\int_{\mathbb{R}^N}w(x)u^{-2\beta}|\nabla \psi|^p dx
\end{equation*}
Replacing $\psi$ by $\psi^\frac{2\beta+p}{p}$ and using Young's inequality with $\theta=\frac{2\beta+p}{2\beta}$ and $\theta'=\frac{2\beta+p}{p}$, we have
\begin{align*}
\int_{\mathbb{R}^N} g(x)(\frac{\psi}{u})^{2\beta+p} \,dx
&\leq c\int_{\mathbb{R}^N}w(x)u^{-2\beta}\psi^{2\beta}|\nabla\psi|^p dx\\
&=c\int_{\mathbb{R}^N} g^{\frac{1}{\theta}}(\frac{\psi}{u})^{2\beta}(g^{-\frac{1}{\theta}}w(x)|\nabla\psi|^p) dx\\
&\leq\epsilon\int_{\mathbb{R}^N} g(x)(\frac{\psi}{u})^{2\beta+p} \,dx+c_{\epsilon}\int_{\mathbb{R}^N} g^{-\frac{\theta'}{\theta}}w^{\theta'}|\nabla\psi|^{p\theta{'}} \,dx
\end{align*}
Therefore we get the inequality
$$
\int_{\mathbb{R}^N} g(x)(\frac{\psi}{u})^{2\beta+p} \,dx
\leq c\int_{\mathbb{R}^N} w^{\theta'}g^{-\frac{\theta'}{\theta}}|\nabla\psi|^{p\theta'}dx
$$
Since $g^{-1}\in L^{\infty}(\mathbb{R}^N)$, we have
$$
\int_{\mathbb{R}^N} g(x)(\frac{\psi}{u})^{2\beta+p} \,dx
\leq c\int_{\mathbb{R}^N} w^{\theta'}|\nabla\psi|^{p\theta'}dx
$$
\end{proof}

\begin{proof}[Proof of Theorem \ref{Main1}]
By contradiction, let us suppose that $u$ be a positive stable solution to the problem $(\ref{mp})_e$. Then by Theorem $\ref{Main}$, we have
$$
\int_{\mathbb{R}^N} g(x)(\frac{\psi}{u})^{2\beta+p} \,dx
\leq c\int_{\mathbb{R}^N} w^{\theta'}|\nabla\psi|^{p\theta'}dx
$$ 
where $\theta=\frac{2\beta+p}{2\beta}$ and $\theta_{'}=\frac{2\beta+p}{p}$.
Choosing $\psi\in C_{c}^{1}(\mathbb{R}^N)$ such that $0\leq\psi_{R}\leq 1$ in $(\mathbb{R}^N)$, $\psi_{R}=1$ in $B_{R}(0)$, 
and $\psi=0$ in $\mathbb{R}^N\setminus B_{2R}(0)$ with $|\nabla\psi|\leq\frac{c}{R}$ for $c>1$ in $(\ref{ineq1})$, we obtain
\begin{equation}\label{main1ineq}
\int_{B_{R}(0)} g(x)(\frac{1}{u})^{2\beta+p} \,dx\leq cR^{-p\theta{'}}\int_{B_{2R}(0)} w^{\theta'} \,dx
\end{equation}
Since $\theta{'}=\frac{2\beta+p}{p}<T_p$ and $||w||^{T_p}_{L^{T_p}(B_{2R}(0))}=o(R^{\mu_p})$, we have 
$$
\int_{B_{2R}(0)} w^{\theta{'}} \,dx\leq c\big(\int_{B_{2R}(0)} w^{T_p} \,dx\big)^{\frac{\theta{'}}{T_p}}\leq c R^\frac{\theta{'}\mu_p}{T_p}
$$ 
Hence from $(\ref{main1ineq})$, we obtain
$$
\int_{B_{R}(0)} g(x)(\frac{1}{u})^{2\beta+p} \,dx\leq cR^{{\theta'}(\frac{\mu_p}{T_p}-p)}
$$
Since $\mu_p<pT_p$, letting $R\to\infty$, we obtain
$$
\int_{\mathbb{R}^N} g(x)(\frac{1}{u})^{2\beta+p} \,dx=0
$$
which is a contradiction.
\end{proof}
\section*{Acknowledgement:}
The first author and second author are supported by Inspire Faculty Award DST-MATH 2013-029
 and by NBHM Fellowship No: 2-39(2)-2014 (NBHM-RD-II-8020-June 26, 2014) respectively.

Kaushik Bal and Prashanta Garain\\
E-mail: kaushik@iitk.ac.in, pgarain@iitk.ac.in\\
Department of Mathematics and Statistics.\\
Indian Institute of Technology, Kanpur.\\
UP-208016, India\\

\end{document}